\documentclass[12pt]{article}

\usepackage[margin=1in]{geometry}
\usepackage{amsmath}
\usepackage{amsfonts}
\usepackage{MnSymbol}
\usepackage{amsthm}
\usepackage{hyperref}
\usepackage[inline]{enumitem}
\usepackage{authblk}

\newtheorem{thm}{Theorem}[section]
\newtheorem{lem}[thm]{Lemma}

\newtheorem{rem}[thm]{Remark}

\newtheorem{prop}[thm]{Proposition}
\newtheorem{cor}[thm]{Corollary}

\begin{document}

\title{Arithmetical pseudo-valuations associated to Dubrovin valuation rings and prime divisors of bounded Krull domains}

\author{F. Van Oystaeyen}
\author{N. Verhulst}
\affil{Universiteit Antwerpen\\
Department of mathematics \& computer science\\
Middelheimlaan 1\\2020 Antwerpen, Belgium}
\date{}
\maketitle

\abstract{We look at value functions of primes in simple Artinian rings and associate arithmetical pseudo-valuations to Dubrovin valuation rings which, in the Noetherian case, are $\mathbb{Z}$-valued. This allows a divisor theory for bounded Krull domains.}

\section{Introduction}

Primes in simple Artinian rings, introduced by Van Oystaeyen and Van Geel, replace noncommutative valuation rings in skewfield. However, the value functions associated to primes have only been studied in skewfields and arithmetical pseudo-valuations have been studied only for orders containing an order with a commutative semigroup of fractional ideals. 

First we consider primes $(R,P)$ where $R$ is a Goldie ring with a simple Artinian quotient ring and such that $R$ is invariant under inner $A$-automorphisms (this is a property of valuation rings in skewfields). We obtain an arithmetical pseudo-valuation (we will abbreviate this to apv) $v:\mathcal{F}(R)\to \Gamma$ where $\mathcal{F}(R)$ is the partially ordered semigroup of fractional ideals of $R$ and $\Gamma$ is a totally ordered semigroup. We characterize the rings for which $\Gamma$ is a group and establish that if $\Gamma$ is an Archimedean group, $R$ is a Dubrovin valuation. If $\bigcap P^{n}=0$, where $P=J(R)$ is the Jacobson radical of $R$, is added to the assumptions on $R$, then $A$ is a skewfield and $R$ is a valuation ring of $A$.

Next we look at Noetherian Dubrovin valuation rings and establish an apv $v:\mathcal{F}(R)\to\Gamma$ where $\Gamma$ is totally ordered such that $P=\left\{a\in A\mid v(RaR)>0\right\}$ and $R=\left\{a\in A\mid v(RaR)\geq 0\right\}$. The case of non-Noetherian Dubrovin valuations allows $P=P^{2}$ and then $v$ cannot exist (since $v(P)=0$ would follow and then $v(I)=0$ for all $I\in\mathcal{F}(R)$). The condition $\bigcap P^{n}=0$ is enough to characterize Dubrovin valuation rings coming from an apv in the well-described way. We may call these the \emph{valued Dubrovin rings}. Our results establish that Noetherian Dubrovin valuations define a $\mathbb{Z}$-valued apv with all the nice properties and so these are good generalizations of discrete valuation rings in skewfields. In the final section we establish that bounded Krull domains have a divisor theory based on those apvs.

\section{Stable fractional primes}

A \emph{prime} in a ring $R$ is a subring $R'$ with a prime ideal $P\subset R'$ such that for any $x,y$ in $R$, $xR'y\subseteq P$ implies either $x\in P$ or $y\in P$. A prime $(P,R')$ is called \emph{localized} if for any $r\in R\setminus R'$, there exist $x,y\in R'\setminus\left\{0\right\}$ with $xry\in R'\setminus P$.
A prime $(P,R')$ in a ring $R$ is called \emph{strict fractional} if for any $r\neq0$ in $R$ there exist $x,y\in R'$ with $0\neq xry\in R'$, e.g. any localized prime is strict fractional. We refer the interested reader to \cite{theBook} for much more information about primes.

\begin{prop}\label{saisunique}
Let $(P,R')$ be a strict fractional prime in a simple Artinian ring $A$. If $R$ is any semisimple Artinian subring of $A$ and $R\neq A$, then $R'$ is not contained in $R$.  
\end{prop}
\begin{proof}
Assume $R'\subseteq R$ and pick $a\in A\setminus R$. Since $R$ is Noetherian, we may choose $L$ maximal for the property that $Lay\subseteq R$ for some $y\in R$. Since $R$ is semisimple Artinian, we have $R=L\oplus U$ where $U$ is a left ideal and $uay\notin R$ for every $u\in U$ (otherwise $(L+Ru)ay\subseteq R$ entails $u\in L$ which is a contradiction). There exist $x',y'\in R'$ with $0\neq x'uayy'\in R'\subseteq R$. Since $L$ is maximal for the property that $Layy'\subseteq R$, it follows that $x'u\in L$ but $x'u\in U$, so $x'u=0$ contradicting $x'uayy'\neq 0$.
\end{proof}

A ring is said to be a \emph{Goldie ring} is the set of regular elements satisfies the Ore condition and $S^{-1}R$ is a semisimple Artinian ring.

\begin{prop}\label{Risorder}
If $(P,R')$ is a strict fractional prime in a simple Artinian ring and $R'$ is a Goldie ring, then $S^{-1}R'=A$ where $S$ is the set of regular elements in $R'$. 
\end{prop}
\begin{proof}
Let $r\in S$, then $r$ is regular in $A$ because if $ru=0$ then $u\in A\setminus R'$ so there exist $x,y\in R'$ with $0\neq xuy\in R'$. Since $S$ satisfies the Ore condition, there are $x'\in R'$ and $r'\in S$ with $r'x=x'r$ so $r'xuy=x'ruy=0$ hence $xuy=0$ which is a contradiction. Since $A$ is simple Artinian, $r^{-1}\in A$ so $S$ is invertible in $A$ and $R'\hookrightarrow A$ extends to $S^{-1}R'\hookrightarrow A$. Since $S^{-1}R'$ is semisimple Artinian, the preceding proposition implies that $S^{-1}R'=A$.
\end{proof}

\begin{rem}
If $R$ is a Goldie prime ring and $I$ is an essential left ideal of $R$ then $I$ is generated by the regular elements of $I$. (See \cite{McCR}.)
\end{rem}

Consider a strict fractional prime $(P,R)$ of a simple Artinian ring $A$ with $R=A^{P}$, i.e. $R=\left\{a\in A\mid aP\subseteq P\text{ and }Pa\subseteq P\right\}$. We always assume that $R$ is Goldie hence a prime ring and an order of $A$ (by proposition \ref{Risorder}). If $P$ is invariant under innner automorphisms of $A$, we say that $(P,R)$ is an \emph{invariant prime} of $A$.

\begin{rem}
If $(R,P)$ is an invariant prime in $A$, then $R$ is invariant under inner automorphisms of $A$.
\end{rem}
\begin{proof}
Consider $u\in A^{*}$. For $p\in P$ we have $uRu^{-1}p=uRu^{-1}puu^{-1}$ and $u^{-1}pu\in P$ so $Ru^{-1}pu\subseteq P$ and $uRu^{-1}p\subseteq uPu^{-1}\subseteq P$. Hence $uRu^{-1}P\subseteq P$ which implies $uRu^{-1}\subseteq R$. A similar reasoning gives $PuRu^{-1}\subseteq R$.
\end{proof}

By a \emph{fractional $R$-ideal} of $A$ we mean an $R$-bimodule $I\subseteq A$ such that $I$ contains a regular element of $R$ and for some $r\in R$, $rI\subseteq R$. Observe that we may choose $r$ regular since $R$ is an order. We will denote the set of fractional ideals of $R$ by $\mathcal{F}(R)$.

\begin{lem}\label{IJinPthenJIinP}
With assumptions as before, we have:
\begin{enumerate}[label=(\arabic*)]
\item If $u$ is regular and $uI\subseteq R$ then $Iu\subseteq R$. Also: $uI\subseteq P$ if and only if $Iu\subseteq P$.
\item If $I,J\in\mathcal{F}(R)$, then $IJ\subseteq P$ if and only if $JI\subseteq P$.
\item If $I,J\in\mathcal{F}(R)$ then $IJ\subseteq R$ implies $JI\subseteq R$ and vice versa. Moreover, if $J\nsubseteq P$ then $I\subseteq R$ and if $P\nsupseteq I\subseteq R$ then $J\subseteq R$.
\end{enumerate}
\end{lem}
\begin{proof}
\begin{enumerate*}[label=(\arabic*)]
\item If $uI\subseteq R$, then $Iu\subseteq u^{-1}Ru=R$. The other case is similar.\\
\item If $IJ\subseteq P$ then, since $(P,R)$ is a prime, either $I$ or $J$ is in $P$, say $I\subseteq P$. Since $I$ is an ideal it is left essential so it is generated by regular elements. For every regular element $u\in I$ $uJ\subseteq P$ yields $Ju\subseteq u^{-1}Pu=P$, hence $JI\subseteq P$. The case $I\nsubseteq P$ and $J\subseteq P$ is similar.\\
\item Suppose $I,J\in\mathcal{F}(R)$ such that $IJ\subseteq R$. If $I,J\subseteq P$ there is nothing to prove since then $IJ\subseteq P$ and $JI\subseteq P$, so assume $J\nsubseteq P$ ($I\nsubseteq P$ is completely similar). From $PIJ\subseteq P$ we obtain then $PI\subseteq P$ since $(P,R)$ is a prime of $A$, so $I\subseteq A^{P}=R$. Again, $I$ is generated by regular elements since it is left essential and for $u\in I$ regular $uJ\subseteq R$ gives $Ju\subseteq u^{-1}Ru=R$ hence $JI\subseteq R$.
\end{enumerate*}
\end{proof}

\begin{cor}
$P$ is the unique maximal ideal of $R$.
\end{cor}
\begin{proof}
Consider an ideal $I\nsubseteq P$ and a regular element $u$ of $R$ which is in $I$ but not in $P$ (this exists since $I$ is generated by regular elements). Then $Ru=RuR\neq R$ so $u^{-1}\notin R$. From $Ru^{-1}RuR=R$ with $RuR\nsubseteq P$ we obtain $Ru^{-1}R\subseteq R$ which is a contradiction.
\end{proof}

Note that we really showed that every regular element in $R\setminus P$ is invertible in $R$.

\begin{cor}
If $\mathcal{C}(P)=\left\{x\in R\mid x\mod P\text{ regular in }R/P\right\}$ satisfies the Ore condition then it is invertible in $R$, i.e. $Q_{P}(R)=R$ or $R$ is local and $P$ is the Jacobson radical of $R$.
\end{cor}
\begin{proof}
If $\mathcal{C}(P)$ is an Ore set in the prime Goldie ring $R$ which is also an order in a simple Artinian ring $A$, then $\mathcal{C}(P)$ consists of regular elements and since $\mathcal{C}(P)\subseteq R\setminus P$ it consists of invertible elements of $R$. Consequently, the localization of $R$ at $\mathcal{C}(R)$ is equal to $R$. It then follows that $P$ is the Jacobson radical of $R$.
\end{proof}

\begin{prop}\label{Cpore}
If $\bigcap P^{n}=0$ then $\mathcal{C}(P)$ satisifes the Ore condition.
\end{prop}
\begin{proof}
We claim that $1+P$ consists of units. Indeed, consider $1+p$ with $p\in P$ and assume it is not regular, then $r(1+p)=0$ for some $0\neq r\in R$. Then $r=-rp$ yields $r\in\bigcap P^{n}$ hence $r=0$ which is a contradiction. If $c\in\mathcal{C}(P)$ then $\overline{c}$ is regular in $R/P$. We have that $P+Rc$ is essential in $R$ since it contains $P$ hence it is generated by  regular elements. Since $Ru=RuR$ for regular $u$, it follows that $P+Rc$ is a two-sideed ideal of $R$, hence $P+Rc=R$ and $\overline{R}\overline{c}=\overline{R}$ i.e. $\overline{c}$ is invertible. Then there is an $\overline{u}\in \overline{R}$ with $\overline{u}\overline{c}=\overline{1}$ which means $uc\in 1+P$. If $rc=0$ then $uru^{-1}uc=0$ which would contradict the fact that all elements of $1+P$ are units. Consequently, $\mathcal{C}(P)$ consists of $R$-regular elements. For every $r\in R$ and $c\in\mathcal{C}(P)$ we have $cr=crc^{-1}c=r'c$ which gives the left Ore condition and also $rc=cc^{-1}rc=cr'$ which gives the right Ore condition. Therefore $\mathcal{C}(P)$ is an Ore set. 
\end{proof}

\begin{cor}\label{inters=0}
If $\bigcap P^{n}=0$ then $\mathcal{C}(P)$ is invertible in $R$ and $R$ is local with Jacobson radical $P$.
\end{cor}

\begin{cor}
$R/P$ is a skewfield.
\end{cor}
\begin{proof}
If $\overline{a}\in R/P$ is not invertible then it is not regular (cfr. the proof of proposition \ref{Cpore}), say $\overline{a}\overline{s}=0$. Let $\overline{a}= a\mod P$ and $\overline{s}=s\mod P$, then $sa\in P$ implies $(Rs+P)a\subseteq P$. Furthermore, $Rs+P$ is two-sided and it contains $P$ strictly so $Rs+P=R$. This means that $Ra\subseteq P$ so $\overline{a}=0$.
\end{proof}

\begin{prop}
Under assumptions as before, the left ideals of $R$ are totally ordered and every finitely generated left ideal is generated by one regular element. 
\end{prop}
\begin{proof}
Suppose $xy\in P$ with either $x$ or $y$ regular in $A$. We suppose without loss of generality that $x$ is regular, so it is invertible in $A$. We find $xRy=xRx^{-1}xy=Rxy\subseteq P$ so since $(R,P)$ is prime, $x$ or $y$ must be in $P$. Consider now $x$ regular (hence invertible) in $A\setminus R$. Since $x\notin R$, there must be a $p\in P$ with $xp\notin P$ (or $px\notin P$ in which case we argue similarly). Then we have $Rx^{-1}xp\subseteq P$ so $x^{-1}\in P$ since it is $A$-regular. Consider now a finitely generated left ideal $I$ in $R$. By \cite{McCR}, it is generated by $R$-regular elements so it is generated by a finite number of $R$-regular elements say $I=Ru_{1}+\cdots +Ru_{n}$. Since $R$ is Goldie, every $R$-regular element is $A$-regular, so by the preceding statements either $u_{1}u_{2}^{-1}$ or $u_{2}u_{1}^{-1}$ must be in $R$. Suppose the latter (again, in the other case we argue similarly), then $Ru_{2}=Ru_{2}u_{1}^{-1}u_{1}\subseteq Ru_{1}$ which means that $Ru_{1}+Ru_{2}=Ru_{1}$. By induction we find that every finitely generated left ideal is principal and in fact even principal for a regular element. This in turn implies that the finitely generated left ideals are totally ordered by inclusion. Suppose now that $I$ and $J$ are left $R$-ideals with $J\nsubseteq I$. There must be a regular $x\in J\setminus I$ and for every $y\in I$ we have either $yx^{-1}\in R$ which would imply $y\in xR\subseteq J$ or $xy^{-1}\in R$ but this is contradictory since it means $x\in Ry\subseteq I$.
\end{proof}

\begin{cor}
Since every left ideal generated by a regular element is two-sided, every finitely generated left ideal is two-sided.
\end{cor}

\section{Arithmetical pseudo-valuations}

An \emph{arithmetical pseudovaluation} on $R$ as before is a function $v:\mathcal{F}(R)\to\Gamma$ for some partially ordered semigroup $\Gamma$ such that:
\begin{enumerate}[label=(APV\arabic*),ref=(APV\arabic*),leftmargin=15mm]
\item $v(IJ)=v(I)+v(J)$;
\item $v(I+J)\geq\min\left\{v(I),v(J)\right\}$;
\item $v(R)=0$;
\item $I\subseteq J$ implies $v(I)\geq v(J)$.
\end{enumerate}
For more information about airthmetical pseudo-valuations, we refer to \cite{theBook} and \cite{VanGeel}.

\begin{thm}\label{apv}
If $(R,P)$ is as before then there is an arithmetical pseudo-valuation $v:\mathcal{F}(R)\to\Gamma$, where $\Gamma$ is a totally ordered semigroup, such that $P=\left\{a\in A\mid v(RaR)>0\right\}$ and $R=\left\{a\in A\mid v(RaR)\geq 0\right\}$.
\end{thm}
\begin{proof}
Observe that for any $I,J\in\mathcal{F}(R)$ we have $IJ\in\mathcal{F}(R)$ and $I+J\in\mathcal{F}(R)$, moreover for every $a\in A$ we have $RaR\in\mathcal{F}(R)$. Indeed, if $a\in A$ then there is a regular $u\in R$ such that $ua\in R$ since $R$ is an order, then $RuaR=RuRaR\subseteq RaR$ and as an $R$-ideal, $RuaR$ contains a regular element of $R$. If $I$ and $J$ are in $\mathcal{F}(R)$ then $IJ$ contains a regular element and if $uI\subseteq R$ and $vJ\subseteq R$ for regular $u$ and $v$ then $Jv\subseteq R$ so $uIJv\subseteq R$ whence $vuIJ\subseteq vRv^{-1}=R$ with $vu$ regular. For $I+J$ we have $vu(I+J)\subseteq R+vuJ$ with $vuJ=vuv^{-1}vJ\subseteq R$ since $vuv^{-1}\in R$.

For any $I\in\mathcal{F}(R)$ we define $v(I)=(P:I)=\left\{a\in A\mid aI\subseteq P\right\}$ and since $RaRI\subseteq P$ if and only if $IRaR\subseteq P$ this is also equal to $v(I)=\left\{a\in A\mid Ia\subseteq P\right\}$. Note that $v(I)\neq \left\{0\right\}$ because $uI\subseteq R$ for some regular $u\in R$, hence $0\neq Pu\subseteq v(I)$. We also have $v(R)=P$. Put $\Gamma=\left\{v(I)\mid I\in\mathcal{F}(R)\right\}$ and define a partial order $\leq$ by $$v(I)\leq v(J)\quad\Leftrightarrow\quad v(I)\subseteq v(J).$$
Note that if $I\subseteq J$ then $v(I)\geq v(J)$. We claim that $\Gamma$ is in fact totally ordered. Indeed, if $I,J\in\mathcal{F}(R)$ such that $v(I)\nsubseteq v(J)$ and $v(J)\nsubseteq v(I)$ then there is an $a\in A$ with $aI\subseteq P$ but $aJ\nsubseteq P$ and a $b\in A$ with $bJ\subseteq P$ but $bI\nsubseteq P$. Since $P$ is prime, $aJbI\nsubseteq P$ but $RbIaJ\subseteq RbPJ\subseteq RbJ\subseteq P$ yields $RaJbI\subseteq P$ which is a contradiction in view of lemma \ref{IJinPthenJIinP}.

We can define a (not necessarily commutative) operation $+$ on $\Gamma$ by putting $v(I)+v(J)=v(IJ)$. The unit for this operation is $v(R)$. We now verify that $+$ is well-defined. Suppose $v(I)=v(I')$ and $v(J)=v(J')$ and consider $x\in v(IJ)$, then $RxRIJ\subseteq P$ so $RxRI\subseteq v(J)=v(J')$ or $RxRIJ'\subseteq P$. By the same lemma as before, $IJ'RxR\subseteq P$ follows hence $J'RxR\subseteq v(I)=v(I')$ i.e. $I'J'RxR\subseteq P$ which implies $x\in v(I'J')$ and consequently $v(IJ)\subseteq v(I'J')$. The other inclusion can be obtained by the same argument if the roles of $I,J$ and $I',J'$ are interchanged.

We now check that this operation is compatible with $\leq$. Take some $v(I)\geq v(J)$ and consider $v(HI)$ and $v(HJ)$. If $q\in v(HJ)$ then $qHJ\subseteq P$ so $qH\subseteq v(J)\subseteq v(I)$ which implies $qHI\subseteq P$ so $q\in v(HI)$. To prove that $\leq$ is also stable under right multiplication, we consider $q\in v(JH)$. Then $qJH\subseteq P$ or equivalently $JHq\subseteq P$. By lemma \ref{IJinPthenJIinP} $HqJ\subseteq P$ follows so $Hq\subseteq v(J)\subseteq v(I)$ hence $IHq\subseteq P$ i.e. $q\in v(IH)$.

If $v(I)\leq v(J)$ then $aI\subseteq P$ yields $a(I+J)\subseteq P$ since $aJ\subseteq P$, so $v(I+J)\supseteq v(I)=\min\left\{ v(I),v(J)\right\}$. Together with the preceding, this implies that $v$ is an arithmetical pseudo-valuation. The only thing left to prove is that $$R=\left\{a\in A\mid v(RaR)\geq 0\right\}\quad \text{and}\quad P=\left\{a\in A\mid v(RaR)>0\right\}.$$

Suppose $v(RaR)>0=v(R)=P$, then there is some $x\in v(RaR)\setminus P$. Now $xRaR\subseteq P$ gives $a\in P$, so $\left\{a\in A\mid v(RaR)>0\right\}\subseteq P$. If $p\in P$, then $v(RpR)\supseteq R\supsetneq P$, hence $p\in\left\{a\in A\mid v(RaR)>0\right\}$ so $P=\left\{a\in A\mid v(RaR)>0\right\}$. If $a\in A$ is such that $v(RaR)=0$ and $p\in P$ then $v(RaRRpR)=v(RaR)+v(RpR)=v(RpR)>0$ so $RaRRpR\subseteq P$ and therefore $ap\in P$ which implies $a\in R$ since $R=A^{P}$. On the other hand, if $r\in R$ then $PRrR\subseteq P$. Since $RrR$ is generated by regular elements, it follows that $r\in\sum Ru_{i}R$ for a finite set of regular $u_{i}$. Consequently, since $\Gamma$ is totally ordered, $v(RrR)=v(Ru_{i}R)$ where $v(Ru_{i}R)$ has the minimal value among these regular elements. If $v(RrR)<0$ then $v(P)\leq v(PRrR)$ since $PRrR\subseteq P$ and then 
\begin{equation}\label{leqiseq}
v(P)\leq v(PRrR)=v(P)+v(RrR)\leq v(P)
\end{equation}
since $v(RrR)<0$. This means that all $\leq$ in \ref{leqiseq} are actually equalities and in fact $v(P)=v(P)+v(Ru_{i}R)=v(PRu_{i}R)$ so if $aPRu_{i}R\subseteq P$ then also $aP\subseteq P$. By choosing $a=u_{i}^{-1}$  we find $u_{i}^{-1}P\subseteq P$. In a similar fashion we find $Pu_{i}^{-1}\subseteq P$ and consequently $u_{i}^{-1}\in A^{P}=R$ so $v(RrR)=v(Ru_{i}R)=v(R)=0$. which contradicts $v(RrR)<0$. Consequently $R=\left\{a\in A\mid v(RaR)\geq 0\right\}$.
\end{proof}

\begin{prop}\label{gammagroup}
With $R,P$ and $A$ as before, $\Gamma$ is a group if and only if for any fractional $R$-ideal $I$ there is a nonzero $y\in R$ with $yI\subseteq R$ but $yI\nsubseteq P$.
\end{prop}
\begin{proof}
If $\Gamma$ is a group and $I\in\mathcal{F}(R)$ then for some $J\in\mathcal{F}(R)$ we have $v(I)+v(J)=0$ i.e. $v(IJ)=v(JI)=v(R)$. Consequently, $IJP\subseteq P\supseteq PIJ$ so $IJ\subseteq A^{P}=R$. Since $aIJ\subseteq P$ iff $aR\subseteq P$ we have $IJ\nsubseteq P$. Then we can choose a $y\in J$ with $Iy\subseteq R$ but $Iy\nsubseteq P$ which implies $RIRy\subseteq R$ and $RIRy\nsubseteq P$.

Suppose now that there is some $y$ with $yI\subseteq R$ but $yI\nsubseteq P$. For any $x\in v(RyRI)$ we have $RxRRyRI\subseteq P$ which implies $RxR\subseteq P$ and consequently $x\in v(R)$. From $RyRI\subseteq R$ we can deduce $v(R)\subseteq v(RyRI)$ hence $v(R)=v(RyRI)$ which means that $v(RyR)$ is the inverse of $v(I)$. 
\end{proof}

Note that the second part of the proof of the preceding theorem guarantees that every $v(I)$ is also $v(RaR)$ for some $a\in A$.

\begin{lem}
If $\bigcap P^{n}=0$ and $\Gamma$ is a group then $R$ is a Dubrovin valuation.
\end{lem}
\begin{proof}
Applying corollary \ref{inters=0} gives us that $R/P$ is prime Goldie with invertible regular elements, i.e. it is a simple Artinian ring. Consider $q\in A\setminus R$. There exists some $y\in A$ with $RyRRqR\subseteq R$ but $RyRRqR\nsubseteq P$. Then there exists a $z\in RyR$ with $zq\in R\setminus P$ and since $RqRRyR\subseteq R$ but $RqRRyR\nsubseteq P$ we can use a similar construction to find an element $z'$ with $qz'\in R\setminus P$.
\end{proof}

\begin{thm}
If $\Gamma$ is a group and $\bigcap P^{n}=0$, then $R$ is a valuation ring and $A$ is a skewfield.
\end{thm}
\begin{proof}
If $a\in R\setminus P$ then $P+Ra$ is essential, two-sided and contains $P$ so it is equal to $R$. Then, for some $r\in R$ and $p\in P$, we have $1=p+ra$. We have already seen (cfr. proof of proposition \ref{Cpore}) that $1+P$ consist of units, so $ra$ is a unit hence $a$ is a unit. If $q\in A\setminus R$ there there is some $y$ with $yq\in R\setminus P$, so $yq$ is a unit of $R$ hence $q$ is a unit of $A$. Finally, if some $p\in P$ were not invertible, then $Ap\subseteq P$ since no element in $Ap$ is a unit. Then we would have $A(RpR)\subseteq P$, but this would contain some regular $u$ which is invertible in $A$ and $Au\subseteq P$ would give a contradiction. This implies that $A$ is a skewfield and $R$ is an invariant Dubrovin valuation on $A$, so it must be a valuation ring.
\end{proof}

\begin{cor}
If $\Gamma$ is an Archimedean group, then $R$ is a valuation ring.
\end{cor}
\begin{proof}
In view of the preceding proposition we only have to show that $\bigcap P^{n}=0$. Suppose it is not, then $I=\bigcap P^{n}$ is a nonzero ideal. Pick $0\neq b\in I$, then $RbR\subseteq I$ is a fractional ideal, hence there exists an ideal $J\in\mathcal{F}(R)$ with $v(J)+v(RbR)=0$. Then $v(P^{n})+v(RbR)\leq 0$ for any $n$, so $nv(P)+v(RbR)\leq 0$. However, putting $v(RbR)=\gamma$, there must be some $n$ with $nv(p)>-\gamma$ which is a contradiction.
\end{proof}

\begin{prop}
Let $R$ be any order in a simple Artinian ring $A$ and suppose that $v:\mathcal{F}(R)\to\Gamma$ is an apv. Then:
\begin{enumerate}[label=(\arabic*)]
\item $P=\left\{ a\in A\mid v(RaR)>0\right\}$ defines a prime $(P,A^{P})$ with $\left\{a\in A\mid v(RaR)\geq 0\right\}\subseteq A^{P}$.
\item if $v(I)=\left\{a\in A\mid aI\subseteq P\right\}$ and $\Gamma$ is a group, then $\left\{a\in A\mid v(RaR)\geq 0\right\}=A^{P}$.
\end{enumerate}
\end{prop}
\begin{proof}
\begin{enumerate*}[label=(\arabic*)]
\item For $a,b\in P$ we have $v(R(a+b)R)\geq\min\left\{v(RaR),v(RbR)\right\}$ which is strictly positive, so $a+b\in P$. Clearly, $P$ is an ideal in $A^{P}$ and $R\subseteq A^{P}$ since $v(R)=0$ and for all $r\in R$ and $p\in P$ we have $v(RrRRpR)=v(RpR)>0$. If $a,a'\in A$ are such that $aA^{P}a'\subseteq P$ then $aRa'\subseteq P$ hence $v(RaRa'R)>0$. From $v(RaR)+v(Ra'R)>0$ it follows that either $v(RaR)>0$ or $v(Ra'R)>0$, i.e. either $a\in P$ or $a'\in P$. If $v(RaR)\geq0$ for some $a\in A$ then for all $p\in P$ $$v(RaRpR)=v(RaR)+v(Rp)>0\quad\text{and}\quad v(RpRaR)=v(RpR)+v(RaR)>0$$ so $a\in A^{P}$. \\
\item Consider $a\in A^{P}$. $RaR$ is invertible in $\mathcal{F}(R)$ so there is a $J\in\mathcal{F}(R)$ with $v(RaR)+v(J)=0=v(J)+v(RaR)$, hence $v(JaR)=v(RaJ)=0$. If $v(RaR)<0$ then $v(J)>0$ or in other words $J\subseteq P$. But then $a\in A^{P}$ would give $RaRJ\subseteq P$ which implies $v(RaRJ)>0$ which is a contradiction. Therefore $v(RaR)\geq0$ and $A^{P}=\left\{a\in A\mid v(RaR)\geq0\right\}$.
\end{enumerate*}
\end{proof}

\section{Arithmetical pseudo-valuations on Dubrovin valuations}

For primes containing an order with commutative semigroup of fractional ideal, Van Geel (\cite{VanGeel}) introduced artithmetical pseudo-valuations, but this condition is very strong and reduces the applicability in practice to maximal orders and Dubrovin valuations in finite dimensional central simple algebras. For Dubrovin valuations on infinite dimensional csa the semigroup $\mathcal{F}(R)$ need not be commutative. The following facts are known about Dubrovin valuations on simple Artinian rings (see \cite{theBook}):
\begin{enumerate}[label=(D\arabic*),ref=(D\arabic*)]
\item $R$ is a (left and right) Goldie ring and a prime order of $A$;
\item $(P,R)$ is a localized prime of $R$;
\item $P=J(R)$ and it is the unique maximal ideal of $R$. Consequently, $1+P$ consists of units;
\item $\mathcal{F}(R)$ is linearly ordered;
\item finitely generated $R$-submodules of $A$ are cyclic.
\end{enumerate}

\begin{prop}\label{okfordv}
For a Noetherian Dubrovin valuation $R$ we have for all $I,J\in\mathcal{F}(R)$ that $IJ\subseteq P$ iff $JI\subseteq P$. 
\end{prop}
\begin{proof}
From $IJ\subseteq P$ it follows that either $I\subseteq P$ or $J\subseteq P$ (by (D2)), assume without loss of generality $I\subseteq P$. By (D4), if $JI\nsubseteq P$ then $P\subsetneq JI$ so we have $R\subseteq JI\subseteq JP\subseteq J$ hence $J=RJ\subseteq JIJ\subseteq JP\subseteq J$ which gives $J=JP$. Since $R$ is an order, there is some regular $u\in R$ with $uJ\subseteq R$ and since $R$ is Noetherian $uJ=\sum a_{i}R$ for a finite set of $a_{i}$'s in $uJ$. Then also $J=\sum u^{-1}a_{i}R$, so $J$ is a finitely generated $R$-submodule of $A$. By Nakayama's lemma $J$ must be zero, which is a contradiction.
\end{proof}

\begin{cor}\label{tologroup}
If $R$ is a Noetherian Dubrovin valuation then $$v:\mathcal{F}(R)\to\Gamma: I\mapsto (P:I)=\left\{a\in A\mid aI\subseteq P\right\}$$ is an arithmetical pseudo-valuation and $\Gamma$ is a totally ordered group. Furthermore, $P=\left\{a\in A\mid v(RaR)>0\right\}$ and $R=\left\{a\in A\mid v(RaR)\geq 0\right\}$.
\end{cor}
\begin{proof}
Using the preceding proposition instead of \ref{IJinPthenJIinP} we can repeat the proof of theorem \ref{apv}. The only thing we need to prove is that $\Gamma$ is a group, so consider $I\in\mathcal{F}(R)$. By a similar argument as in the proof of proposition \ref{okfordv} it is finitely generated as a left $R$-ideal of $A$. By (D5) it is cyclic, in fact $I=Ru$ for some regular $u$ and thus $RuR=Ru$. Since $R$ is a Dubrovin valuation, there is some $a\in A$ with $ua\in R\setminus P$. Then $Rua\subseteq R$, $Rua\nsubseteq P$ and $Ia\subseteq R$, $Ia\nsubseteq P$. We can similarly find a $b$ such that $bI\subseteq R$ but $bI\nsubseteq P$. We can now repeat the last part of the proof of proposition \ref{gammagroup} to conclude that $\Gamma$ is a group.
\end{proof}

\begin{rem}
If $R$ is a Dubrovin valuation where $$v:\mathcal{F}(R)\to\Gamma: I\mapsto\left\{a\in A\mid aI\subseteq P\right\}$$ is a non-trivial arithmetical pseudo-valuation with values in a totally ordered group, then $IJ\subseteq P$ if and only if $JI\subseteq P$. Indeed, suppose $IJ\subseteq P$ and $JI\nsubseteq P$ then, as in proposition \ref{okfordv} we find $JP=J$ but then $v(P)=0$ which is impossible.
\end{rem}

If $R$ is non-Noetherian, then $P=P^{2}$ is possible in which case no nice apv can exist since otherwise $v(P)=2v(P)$ which would imply $v(P)=0$. If we exclude this slightly pathological case, a nice apv does exist.

\begin{prop}
Let $R$ be a Dubrovin valuation with $\bigcap P^{n}=0$, then there is an apv as before.
\end{prop}
\begin{proof}
If $I,J\in\mathcal{F}(R)$ with $IJ\subseteq P$ but $JI\nsubseteq P$. The same argument as in proposition \ref{okfordv} leads to $J=JP$ so $J=JP^{n}$ for any $n$. There is some regular $u\in R$ with $uJ\subseteq P$ hence $uJ=uJP^{n}\subseteq P^{n+1}$. But then $uJ=0$ which implies $J=0$ and this is a contradiction. Now we can proceed as in corollary \ref{tologroup} to find an apv with values in a semigroup. 

The only thing we need to prove is that $\Gamma$ is a group. Lemma 1.5.4 in \cite{theBook} says that $P=Rp=pR$ for some regular $p\in P$. Since $P$ is principal as a left $R$-ideal, lemma 1.5.6 in the same source gives $PP^{-1}=R=P^{-1}P$ (here $P^{-1}=\left\{a\in A\mid PaP\subseteq P\right\}$). Consider now a fractional $R$-ideal $I$. Clearly, $(R:I)I\subseteq R$. Suppose we also have $(R:I)I\subseteq P$, then $P^{-1}(R:I)I\subseteq R$ hence $P^{-1}(R:I)\subseteq (R:I)$ so $P^{-1}(R:I)I\subseteq P$. This means $(R:I)I\subseteq P^{2}$ and by repeating this process we find $(R:I)I\subseteq P^{n}$ for any $n$, but $(R:I)I\subseteq \bigcap P^{n}=0$ which is a contradiction. Therefore, $(R:I)I\subseteq R$ but $(R:I)I\nsubseteq P$, so there exists an $a\in (R:I)$ such that $aI\subseteq R$ but $aI\nsubseteq P$.
\end{proof}

The following characterizes Noetherian Dubrovin valuation rings within the class of rank one Dubrovin valuations. The result may be known but we found no reference for it in the literature. Recall that the rank of a Dubrovin valuation ring is the maximal length of a chain of Goldie prime ideals in the ring. A rank $1$ Dubrovin valuation ring on a simple Artinian $A$ is a maximal subring of $A$.

\begin{prop}\label{NoeDub}
For a Dubrovin valuation $R$ on a simple Artinian ring $A$ the following are equivalent:
\begin{enumerate}[label=(\arabic*),ref=(\arabic*)]
\item $R$ is Noetherian. \label{(1)}
\item $R$ has rank $1$ and $P\neq P^{2}$. \label{(2)}
\item $R$ has rank $1$ and $\bigcap P^{n}=0$. \label{(3)}
\end{enumerate}
\end{prop}
\begin{proof}
\begin{enumerate*}[leftmargin=*]
\item[$(1)\Rightarrow(2)$] If $R$ is Noetherian then all ideals (and $R$-ideals of $A$) are principal, so $P\neq P^{2}$. Suppose $0\neq Q$ is another prime ideal in $P$. Let $P=Rp$, then $Q=Ip$ for some non-trivial ideal $I$ of $R$. $Q=IP$ yields $I\subseteq Q$ since $Q$ is prime and $P\nsubseteq Q$. Hence $Q=IP\subseteq QP\subseteq Q$ implies $Q=QP$ which implies $Q=0$ by Nakayama's lemma.

\item[$(2)\Rightarrow(3)$] If $P\neq P^{2}$, then $\bigcap P^{n}\neq P$ which, by Lemma 1.5.15 in \cite{theBook}, gives $\bigcap P^{n}=0$.

\item[$(3)\Rightarrow(1)$] Since $\bigcap P^{n}=0$, $P\neq P^{2}$. Since $R$ is rank $1$, $R=O_{l}(I)=O_{r}(I)$ for any $R$-ideal $I$. By proposition 1.5.8 in \cite{theBook}, it follows that if $I$ is not principal, then $II^{-1}=P$ and $P=P^{2}$ which s a contradiction. 
\end{enumerate*}

\end{proof}

\section{Divisors of Bounded Krull orders}

We consider a prime Noetherian ring $R$. It is an order in a simple Artinian ring $A=Q(R)$, the classical ring of fractions. If $R$ is a maximal order then the set of divisorial $R$-ideals of $A$ (cfr. \cite{theBook}) is a group and since $R$ is Noetherian it is an Abelian group generated by the maximal divisorial ideals and every maximal divisorial ideal is a minimal prime ideal. 

Recall that an order is an \emph{Asano order} if every ideal $I\neq 0$ of $R$ is invertible and it is a \emph{Dedekind order} if it is an hereditary Asano order. If $R$ is an Asano order satisfying the ascending chain condition on ideals, then $\mathcal{F}(R)$ is the Abelian group generated by maximal ideals and every maximal ideal is a minimal nonzero prime ideal. Any bounded Noetherian order is a Dedekind order.

A semi-local order $R$ in a simple Artinian $A$ is a Noetherian Asano order if and only if it is a principal ideal ring. If $R$ is a Dubrovin valuation ring of $A$ then $R$ is a maximal order if and only if $\mathrm{rk}(R)=1$ and $R$ is Asano if and only if it is a principal ideal ring, so a Noetherian Dubrovin valuation ring is a Noetherian maximal order and an Asano order, i.e. a principal ideal ring.

\begin{prop}
If $R$ is a Noetherian Dubrovin valuation then the corresponding apv takes values in $\mathbb{Z}$.
\end{prop}
\begin{proof}
Since $R$ is a Noetherian Asano order, $\mathcal{F}(R)$ is generated by the maximal ideals of $R$, but since $P$ is the unique maximal ideal and the value group is necessarily torsion-free, we have $\mathcal{F}(R)=\mathbb{Z}$.
\end{proof}

Recall that an order $R$ in a simple Artinian ring $A$ is said to be a Krull order if it is a maximal order and it is $\tau$-Noetherian (see \cite{theBook}, definition 2.2.2). A Noetherian order $R$ in a simple Artinian $A$ is a Krull order if and only if it is a maximal order.

\begin{thm}
Let $R$ be a prime Noetherian ring and an order in $A=Q(R)$. Suppose every minimal nonzero prime ideal is localizeable, $R_{p}$ is a Dubrovin valuation ring for every $p\in X^{1}(R)=\left\{\text{minimal nonzero prime ideals of }R\right\}$ and $R=\bigcap R_{p}$, then $R$ is a bounded Krull domain.
\end{thm}
\begin{proof}
Since $R$ is Noetherian, every $R_{p}$ is Noetherian too and since it is a Dubrovin valuation it must be an Asano order hence a principal ideal ring. Since every $R_{p}$ is a maximal order, so is $R$. As a Noetherian maximal order, $R$ is a Krull domain and by theorem 2.2.16 in \cite{theBook} every regular element is a non-unit in only finitely many of the $R_{p}$'s (for $p$ maximal divisorial, i.e. $p\in X^{1}(R)$). Theorem 2.2.20 in the same source gives the result.
\end{proof}

\begin{rem}
If $R$ is a bounded Krull domain then $R=\bigcap_{p\in X^{1}(R)} R_{p}$ (cfr. \cite{theBook} 2.2.18 \& 2.2.19). Then every $R_{p}$ is rank one Dubrovin valuation with $P\neq P^{2}$ (2.2.16 in the aforementioned source), so by proposition \ref{NoeDub} it is Noetherian. This gives a correspondence between $p\in X^{1}(R)$ and apvs.
\end{rem}

A \emph{divisor} of a bounded Krull domain $R$ is an element in the free Abelian group $\bigoplus_{p\in X^{1}(R)}\mathbb{Z}p$. To any $I\in\mathcal{F}(R)$ we can associate the divisor $\mathrm{div}(I)=\sum v_{p}(I_{p})p$ where $I_{p}=R_{p}I$. This definition is justified by the following:

\begin{prop}
Suppose $R_{p}$ is Noetherian. If $I$ is an $R$-ideal of $A$, then $I_{P}$ is an $R_{P}$-ideal of $A$.
\end{prop}
\begin{proof}
Let $u$ be regular in $R$ with $uI\subseteq R$, then $RuRI\subseteq R$ and $RuR=Ru'$ for some regular $u'\in R$. Then $R_{p}u'I$ is the localization of $Ru'I$ and it is an ideal of $R_{p}$. $R_{p}u'$ is the localization of $Ru'$ so it is also an ideal of $R_{p}$, hence $R_{p}u'I=R_{p}u'R_{p}I=R_{p}u'I_{p}$ is an ideal of $R_{p}$. Now $U'I_{p}\subseteq R_{p}$, i.e. $I_{p}$ is an $R_{p}$-ideal of $A$.
\end{proof}

Observe that since any regular element is a non-unit in only finitely many localizations, $\mathrm{div}(I)$ contains only finitely many non-zero terms. Moreover, $\mathrm{div}(I)\leq\mathrm{div}(J)$ is and only if $v_{p}(I)\leq v_{p}(J)$ for all $p\in X^{1}(R)$. By putting $I^{*}=\bigcap_{p\in X^{1}(R)}I_{p}$ we find $v_{p}(I)=v_{p}(I*)$. For bounded Krull domains we can consider $\mathrm{div}:\mathcal{F}(R)\to\mathrm{Div}(R)$. This is a group morphism of Abelian groups and it is order reversing in the sense that $I\subseteq J$ yields $\mathrm{div}(I)|\mathrm{div}(J)$. Further divisor theory requires a version of the aproximation theorems. This is work in progress.

\end{document}